\documentclass[a4paper]{amsproc}
\usepackage{amsmath}
\usepackage{amssymb}
\usepackage{amsthm}
\usepackage{dsfont}
\usepackage{stmaryrd}
\usepackage[mathcal]{eucal}
\usepackage{verbatim} 
\usepackage{tikz-cd}
\usepackage{array}
\usepackage{mathalfa}
\usepackage{upgreek}
\usepackage{enumerate}
\usepackage{enumitem}
\usepackage{hyperref}
\usepackage{setspace}
\usepackage{caption}
\usepackage{filecontents}
\usetikzlibrary{shapes.geometric, arrows}

\renewcommand{\leq}{\leqslant}
\renewcommand{\geq}{\geqslant}

\title[An Algebraic Characterization of the Point-Pushing Subgroup]{An Algebraic Characterization of the Point-Pushing Subgroup}

\subjclass[2010]{Group theory and generalizations}


\author[Akin]{\bfseries Victoria Akin}

\address{
Department of Mathematics \\ 
University of Chicago  \\
Chicago, IL\\
60637}
\email{toriakin@math.uchicago.edu}

\newcommand{\bbZ}{\mathbb{Z}}  
\newcommand{\bbQ}{\mathbb{Q}}  
\newcommand{\bbC}{\mathbb{C}}
\newcommand{\bbN}{\mathbb{N}}
\newcommand{\puncture}{\Sigma_{g,1}}
\newcommand{\surface}{\Sigma_{g}}
\newcommand{\boundary}{\Sigma_{g}^1}
\newcommand{\mcg}{\text{Mod}(\Sigma_{g,1})}
\newcommand{\extmcg}{\text{Mod}^\pm(\Sigma_{g,1})}
\newcommand{\extmcgs}{\text{Mod}^\pm(\Sigma)}
\newcommand{\mcgs}{\text{Mod}(\Sigma)}
\newcommand{\mcgn}{\text{Mod}(\hat{\Sigma}_{g,1})}
\newcommand{\mcgb}{\text{Mod}(\Sigma_{g}^1)}

\newcommand{\al}{\alpha}
\newcommand{\pp}{\pi_1(\Sigma_g)}
\newcommand{\gl}{\text{GL}_{2g}(\bbZ)}
\newcommand{\gln}{\text{GL}_{n}(\bbZ)}
\newcommand{\glnq}{\text{GL}_{n}(\bbQ)}
\newcommand{\Sp}{\text{Sp}_{2g}(\bbZ)}
\newcommand{\Sppm}{\text{Sp}_{2g}^{\pm}(\bbZ)}
\newcommand{\Spq}{\text{Sp}_{2g}(\bbQ)}
\newcommand{\ia}{\mathcal{I}}
\newcommand{\iap}{\ia(\Sigma)}
\newcommand{\ian}{\ia(N)}
\newcommand{\iant}{\ia_2(N)}
\newcommand{\iapt}{\ia_2(\Sigma)}
\newcommand{\iapk}{\ia_k(\Sigma)}
\newcommand{\iank}{\ia_k(N)}
\newcommand{\iast}{\ia_{g,1}}
\newcommand{\iab}{\ia_{g}^{1}}

\newcommand{\mcgone}{\text{Mod}(\Sigma_{1,1})}
\newcommand{\mcgnpm}{\text{Mod}^\pm(\hat{\Sigma}_{g,1})}
\newcommand{\mcgg}{\text{Mod}(\Sigma_g)}
\newcommand{\po}{\pi_1}

\tikzstyle{startstop} = [rectangle, rounded corners, minimum width=3cm, minimum height=1cm,text centered, draw=black,]
\usepackage{amsthm}
  \newtheoremstyle{TheoremNum}
        {\topsep}{.2cm}              
        {\itshape}                      
        {}                              
        {\bfseries}  
        {}                             
        { }                             
        {\thmname{#1}\thmnote{ #3}}
    \theoremstyle{TheoremNum}
    \newtheorem{corn}{Corollary}
    \newtheorem{propn}{Proposition}
    
    \newtheoremstyle{Theorembold}
            {.2cm}{.2cm}              
            {\itshape}                      
            {}                              
            {\bfseries}                     
            {.}                             
            { }                             
            {\thmname{#1}\thmnumber{ #2}\thmnote{ #3}}
            
                \newtheoremstyle{TheoremboldDef}
                        {.2cm}{.2cm}              
                        {}                      
                        {}                              
                        {\bfseries}                     
                        {.}                             
                        { }                             
                        {\thmname{#1}\thmnumber{ #2}\thmnote{ #3}}
   
    \newtheoremstyle{TheoremboldRem}
            {.5cm}
            {}              
            {}                      
            {}                              
            {\bfseries}                     
            {.}                             
            { }                             
            {\thmname{#1}\thmnumber{ #2}\thmnote{ #3}}

\theoremstyle{Theorembold}
 \newtheorem{thm}{Theorem}[section]
 \newtheorem{prop}[thm]{Proposition}
 \newtheorem{lem}[thm]{Lemma}
 \newtheorem{cor}[thm]{Corollary}
\theoremstyle{TheoremboldDef}
 \newtheorem{exm}[thm]{Example}
 \newtheorem{dfn}[thm]{Definition}
  \newtheorem{rem}[thm]{Remark}
  \theoremstyle{TheoremboldRem}
 \newtheorem*{remarks}{Remarks}
 \newtheorem*{cor*}{Corollary}
 \theoremstyle{TheoremboldRem}
  
 \usepackage{chngcntr}

\begin{document}

\vspace{18mm} \setcounter{page}{1} \thispagestyle{empty}

\begin{abstract}
The \emph{point-pushing subgroup} $P(\surface)$ of the mapping class group $\mcg$ of a surface with marked point is an embedding of $\pp$ given by pushing the marked point around loops. We prove that for $g\geq 3$, the subgroup $P(\surface)$ is the unique normal, genus $g$ surface subgroup of $\mcg$. As a corollary to this uniqueness result, we give a new proof that $\text{Out}(\text{Mod}^\pm(\Sigma_{g,1}))=1$, where $\text{Out}$ denotes the outer automorphism group; a proof which does not use automorphisms of complexes of curves. Ingredients in our proof of this characterization theorem include combinatorial group theory, representation theory, the Johnson theory of the Torelli group, surface topology, and the theory of Lie algebras.
\end{abstract}

\maketitle

\section*{Introduction}
Let $\Sigma_{g}$ (respectively, $\Sigma_{g,1}$) be a compact, connected surface of genus $g$ (respectively, with one marked point). Let $\Sigma$ be either $\Sigma_g$ or $\Sigma_{g,1}$. The \emph{mapping class group} $\mcgs$ is the group of orientation-preserving homeomorphisms of $\Sigma$ modulo isotopy. The map $\puncture\to\surface$ given by ``forgetting" the marked point induces an injection \[F: \mcg\hookrightarrow\mcgg.\] For $g\geq 2$, the \emph{point-pushing subgroup} is defined by \[P(\surface):= \ker(F).\] Informally, $P(\surface)$ is the subgroup of $\mcg$ consisting of elements that ``push" the marked point along closed curves in the surface. Birman in \cite{birmancraggs} (see also \cite{birman2}) proved that $P(\surface)\cong\pp$. A \emph{genus $h$ surface group} is any group isomorphic to $\pi_1(\Sigma_h)$. In particular, $P(\surface)$ is an example of a normal, genus $g$ surface subgroup of $\mcg$.
\begin{thm}[(Uniqueness of $P(\surface)$)]\label{maintheorem}
Let $g\geq 3$. The point-pushing subgroup $P(\surface)$ is the unique normal, genus $g$ surface subgroup inside $\mcg$.
\end{thm}

\begin{remarks}[on Theorem \ref{maintheorem}]\ \\
\vspace{-.5cm}
\begin{enumerate}[label=\arabic*.,leftmargin=1.5cm]
\item  Theorem \ref{maintheorem} has a beautiful free group analogue, proven by Formanek in 1990. Our proof follows in outline the proof given by Formanek in \cite{Formanek}. Even so, in establishing the main result we will have to overcome several obstacles to reconcile the differences between free groups and surface groups.\\
 Let $\{a_1,b_1,\dots, a_g,b_g\}$ be a standard set of generators for $\pp$. In general, the surface relation $(\Pi_i^g[a_i, b_i]=1)$ pervades the objects associated to $\pp$ and muddies the analogy between free and surface groups. Some key differences between $F_n\triangleleft\text{Aut}(F_n)$ and $\pp\triangleleft\mcg$ are summarized in the following table:
 
 \hspace{-2.5cm}\begin{tabular}[h]{p{6.8cm}p{.8cm}p{6.7cm}}
 &&\\
 $F_n\triangleleft\text{Aut}(F_n)$  && $\pp\triangleleft\mcg$ \\
 \hline &&\\ The representation theory of $\glnq$ reveals properties of $\text{Aut}(F_n)$ because &&
  The representation theory of $\Spq$ reveals properties of $\mcg$ because\\ $\text{Aut}(F_n)\twoheadrightarrow \text{Aut}(F_n/\gamma_2(F_n))\cong \gln.$ && $\mcg\twoheadrightarrow\Sp.$\\
 &&\\

 Let $\ia(F_n)$ be the Torelli subgroup of $\text{Aut}(F_n)$, see Definition \ref{torelli}. $\ia(F_n)$ has torsion-free abelianization. Specifically, $H_1(\ia(F_n);\bbZ)\cong\Lambda^3\bbZ^n$. && Let $\iap$ be the Torelli subgroup of $\mcg$, see Definition \ref{torelli}. The abelianization of $\iap$ contains 2-torsion. That is, $H_1(\iap);\bbZ)\cong \Lambda^3\bbZ^{2g}\oplus \mathcal{B}/\langle \al \rangle$ where $\mathcal{B}/\langle\al\rangle$ is 2-torsion, see Proposition \ref{computeH1}. The existence of this 2-torsion comes from the Rochlin invariant in 3-manifold theory.\\ &&\\ Let $\ia_2(F_n)$ be the second term in the Andreadakis-Johnson filtration, see Definition \ref{torelli}. Then
 $[\ia(F_n),\ia(F_n)]=\ia_2(F_n)$. && Let $\iapt$ be the second term in the Andreadakis-Johnson filtration, see Definition \ref{torelli}. Then $[\iap,\iap]\not=\iapt$.
\\
 \end{tabular}\\
 \\
 
\item That $P(\surface)$ is normal in $\mcg$ is necessary for the uniqueness result stated in Theorem \ref{maintheorem}. Clay-Leininger-Mangahas in \cite[Cor.1.3]{nonnormal} construct infinitely many nonconjugate genus $g$ surface subgroups in $\mcgg$. See also the work of Leininger-Reid \cite[Cor.5.6]{leininger}. Specifically for a surface with one marked point, we can find surface subgroups $\pi_1(\Sigma_h)<\mcg$ for infinitely many $h$ using the Thurston norm (see Example \ref{example} below). \\
\item Theorem \ref{maintheorem} does not hold for $g=1$. Because $\mcgone\cong SL_2\bbZ$ has a finite index free subgroup, $\mcgone$ has no surface subgroups. It is not known whether or not $P(\Sigma_2)$ is the only normal, genus 2 surface subgroup in $\text{Mod}(\Sigma_{2,1})$.\\
\end{enumerate}

\end{remarks}

The \emph{extended mapping class group} $\extmcgs$ is the group of all homeomorphisms (orientation preserving and reversing) of $\Sigma$, modulo isotopy. The Dehn-Nielsen-Baer theorem establishes an isomorphism \[\Phi:\extmcg\stackrel{\cong}{\longrightarrow} \text{Aut}(\pp).\] As a consequence of the Dehn-Nielsen-Baer theorem, \[\Phi(P(\surface))=\text{Inn}(\pp)\triangleleft \text{Aut}(\pp)\] where $\text{Inn}(\pp)$ is the group of inner automorphisms of $\pp$.  \\

Burnside in \cite[pp. 261]{burnside} proved that for a centerless group $G$ (which gives $G\cong\text{Inn}(G)\triangleleft\text{Aut}(G)$), if every $\phi\in \text{Aut}(\text{Aut}(G))$ satisfies $\phi(G)=G$, then \[\text{Aut}(\text{Aut}(G))=\text{Inn}(\text{Aut}(G))\cong \text{Aut}(G).\] See Section \ref{burnsec} below for a short proof. Since $\pp$ is centerless for $g>1$, Burnside's result together with Theorem \ref{maintheorem} implies: \\ 

\begin{cor}[(Ivanov-McCarthy's Theorem)]\label{maincorollary}
Let $g\geq 3$. Then $\text{Out}(\extmcg)$ is trivial.
\end{cor}

\begin{remarks}[on Corollary \ref{maincorollary}]\ \\
\vspace{-.4cm}
\begin{enumerate}[label=\arabic*.,leftmargin=1.5cm]
\item For $g\geq 3$, Ivanov-McCarthy proved that Out($\extmcg$)=1, from which they deduced that $\text{Out}(\mcgg)\cong \bbZ/2\bbZ$, see \cite{IvanovMccarthy}, \cite[Th.5]{ivanov1984} and \cite[Th.1]{McCtrivout}. In fact, Ivanov-McCarthy proved a much stronger result for injective homeomorphisms of finite index subgroups of $\mcgg$. Their work uses the deep theorem of Ivanov that the automorphism group of the complex of curves is the extended mapping class group. Our proof does not use this theorem. \\
\item The result of Ivanov-McCarthy that Out($\extmcg$)=1 implies that $P(\surface)$ is characteristic in $\mcg$, since all automorphisms of $\extmcg$ are inner.  In contrast, our characterization theorem (Theorem \ref{maintheorem}) implies that $P(\surface)$ is characteristic, from which we deduce (with Burnside) that $\mcg$ has no outer automorphisms.\\
\item McCarthy in \cite{McCtrivout} proved that $\text{Out(Mod)}^\pm(\Sigma_{2,1})$ is nontrivial, which implies (with Burnside) that $P(\Sigma_{2})$ is not characteristic in $\text{Mod}^\pm(\Sigma_{2,1})$. Thus, $P(\Sigma_2)$ is not the only normal, genus 2 surface subgroup in $\text{Mod}^\pm(\Sigma_{2,1})$. However, it is unknown whether or not these additional normal, genus 2 surface subgroups are contained in $\text{Mod}(\Sigma_{2,1})$.\\
\end{enumerate}
 \end{remarks}

\subsection{Structure of the Proof} Ingredients in our proof of Theorem \ref{maintheorem} include combinatorial group theory, representation theory, the Johnson theory of the Torelli group, the theory of Lie algebras, and surface topology. These tools allow us to characterize $P(\surface)$ in terms of two filtrations: the lower central series of $P(\surface)$ and the Andreadakis-Johnson filtration of $\mcg$. By showing any arbitrary normal, genus $g$ surface subgroup must also have those same characterizing properties, we demonstrate that $P(\surface)$ is unique.\\

 To condense notation, let $P:=P(\surface)$. Let $N\triangleleft\mcg$ be a normal subgroup abstractly isomorphic to $\pp$. We must prove that $N=P$.\\

  \begin{dfn}\label{lcs}
 The \emph{lower central series} of a group $G$, denoted as \[G=\gamma_1(G)\supset\gamma_2(G)\supset\dots\] is defined inductively as $\gamma_{i+1}(G)=[\gamma_i(G),G]$.
  \end{dfn}
  Let $Z(G)$ denote the center of a group $G$. The lower central series is \emph{central}, i.e. $\gamma_k(G)/\gamma_{k+1}(G)\subset Z(G/\gamma_{k+1}(G))$ for each $k$. Further, each $\gamma_k(\pp)$ is \emph{characteristic} in $\pp$, i.e. invariant under automorphisms of $\pp$. As such, there is a family of well-defined maps \[\Psi_k:\mcg\to\text{Aut}(\pp/\gamma_{k+1}(\pp)).\]
   \begin{dfn}\label{torelli}
 The \emph{Johnson filtration} of $\mcg$, denoted as \[\iap=\ia_1(\Sigma)\supset\iapt\supset\dots\] is defined as \[\iapk:=\ker(\Psi_k).\]  The first term $\iap$ is referred to as the \emph{Torelli group} of $\mcg$. 
  \end{dfn}
   
 By assumption, $N\cong\pp$. As such, for some surface $\hat{\Sigma}_{g,1}$, we can define an injection $N\hookrightarrow\mcgn$ so that the image of $N$ is the point-pushing subgroup in $\mcgn$. In this paper, we will consider both the Johnson filtration for $\mcgn$ and for $\mcg$. To distinguish these two filtrations we will use the notation $\iank$ for $\mcgn$ and $\iapk$ for $\mcg$.  We will gradually ``push" $P$ and $N$ through the terms of these Johnson filtrations in order to capture salient properties. Eventually, we will establish the following chain of containments:
   \[\gamma_2(N)\subseteq\gamma_2(P)\subseteq\iapt\subseteq\iap\subseteq\ian.\]

  Furthermore, we will give the following useful characterization of $P(\surface)$ in terms of the linear central filtrations defined above.
 \begin{propn}[\ref{mainprop} (Characterization of $P$).]\label{charofP}
 Let $g\geq 3$. Then \[P(\surface)=\{x\in\iap \, |\, [x,\iap]\subset\gamma_2(P(\surface))\}.\]
  \end{propn}
  By proving Proposition \ref{mainprop} we will also characterize $N$ as \[N=\{x\in\ian\, |\,[x,\ian]\subset\gamma_2(N)\}.\] 
   Notice that Proposition \ref{mainprop} together with the two inclusions $\iap\subset\ian$ and $\gamma_2(N)\subset\gamma_2(P)$ implies the following chain of containments:
 \begin{eqnarray*}
 N&=&\{x\in\ian \, |\, [x,\ian]\subset\gamma_2(N)\}\\
 & \subseteq & \{x\in\iap \, |\, [x,\iap]\subset\gamma_2(N)\}\\
 & \subseteq & \{x\in\iap \, |\, [x,\iap]\subset\gamma_2(P)\}\\
 & = & P.
 \end{eqnarray*}
 
 That is, $N\subseteq P$. Applying the index formula $[N:P]\cdot \chi(\Sigma)=\chi(\hat{\Sigma})$, we can conclude that $N=P$.\\
 
In summary, we divide our proof into the following two main parts:
 \begin{itemize}
 \item \textbf{Sections 1-3:} Demonstrate the chain of containments  \[ \gamma_2(N)\subset\gamma_2(P)\subset\iapt\subset\iap\subset\ian.\] 
 \item  \textbf{Section 4:} Characterize the point-pushing subgroup as \[P=\{x\in\iap \, |\, [x,\iap]\subset\gamma_2(P)\}.\]
 \end{itemize} From these two steps, it follows that N=P.

\subsection{Acknowledgements.} I would like to extend a deep thank you to my advisor Benson Farb for suggesting this project, making extensive comments on an earlier draft, and offering continued guidance. I would also like to thank my advisor Jesse Wolfson for his  insight, advice, and comments on an earlier draft. Thank you to Nick Salter for several helpful conversations as well as corrections and comments on an earlier draft. Thank you to Simion Filip, Victor Ginzburg, Sebastian Hensel,  Dan Margalit,  MurphyKate Montee, Chen Lei, Andy Putman, Emily Smith, and  Wouter van Limbeek for helpful comments and conversations.\\

The following Sections 1 through 5 give a proof of Theorem \ref{maintheorem}. 
\section*{1. Action on homology: $N\subset \iap$.}
\addtocounter{section}{1}
\setcounter{subsection}{0} As above, let $P$ be the point-pushing subgroup of $\mcg$. Let $N\triangleleft\mcg$ be abstractly isomorphic to $\pp$. While $N$ need not act as the point-pushing subgroup on $\surface$, we can choose $N$ to be the point-pushing subgroup of $\mcgn$ for some surface $\hat{\Sigma}_{g,1}$. We will use $\iapk$ to denote the Johnson filtration for $\mcg$, and we will use $\iank$ to denote the Johnson filtration for $\mcgn$. For the remainder of the paper, let $g\geq 3$. \\

 In this section, we will work toward establishing the chain of containments \[\gamma_2(N)\subset\gamma_2(P)\subset\iapt\subset\iap\subset\ian\]   by proving that that $\iap\subset\ian$ and $N\subset\iap$.\\
 
 Let $\beta$ be a loop in $\puncture$ based at the marked point, $x_0$. This loop defines an isotopy from the marked point to itself which can be extended to all of $\puncture$. (For a more precise explanation see e.g. \cite[Setc.4.2]{primer}.) Denote this homeomorphism by $\phi_\beta$. The point-pushing subgroup $ P\triangleleft\mcg$ is exactly the subgroup of isotopy classes of homeomorphisms of the form $\phi_\beta$ for any based loop $\beta$. Let $[\beta]\in \pp$ be the homotopy class of loops containing $\beta$. There is a well-defined map 
						\[Push:\pp\to P\] 
							given by 
						\[Push([\beta])= [\phi_\beta].\]
Birman in \cite{birman} proved that the map $Push$ is an isomorphism.  Because $P$ is normal, $\mcg$ acts on $P$ via conjugation. Alternately, the action of $\mcg$ on $\puncture$ induces an action on the fundamental group $\pp$. The map $Push$ respects the action of $\mcg$. That is, for $\psi\in\mcg$ and $[\beta]\in \pp$
				\begin{equation*}
			\tag{$\dagger$} Push(\psi_{\ast}([\beta]))=\psi Push([\beta])\psi^{-1}.	\end{equation*}
For convenience, we will sometimes equate $P$ with $\pp$. For full details regarding the isomorphism between $P$ and $\pp$ see Section $4.2$ of \cite{primer}.\\

 The point-pushing subgroup $P$ acts by free homotopies on the unmarked surface $\Sigma_g$. As such, $P$ acts trivially on $H_1(\pp;\bbZ)\cong \pp/\gamma_2(\pp)$. That is, $P\subset\iap$. We want to show that $N$ also has the property $N\subset\iap$.\\

Given $\phi\in\mcg$ and $n\in N$, define the map \[\al:\mcg\to\text{Aut}^\pm(N)\cong \mcgnpm\] by \[\al(\phi)(n)=\phi n \phi^{-1}.\] We have the following exact sequences for $P$ and $N$:
 \[
\begin{tikzcd}
1 \arrow{r}{}  & \iap \arrow{r}{} & \mcg \arrow{d}{\al} \arrow{r}{\Psi_\Sigma} & \Sp\arrow[dotted]{d}{\bar{\al}} \arrow{r} & 1\\
1 \arrow{r}{} & \ian\arrow{r} & \text{Aut}^\pm(N) \arrow{r}{\Psi_N} & \Sppm\arrow{r} & 1
\end{tikzcd}\tag{$\ddagger$}
\]
where $\Sppm$ is the subgroup of $\gl$ generated by $\Sp$ and the image of any orientation-reversing homeomorphism. Because all orientation reversing homeomorphisms are nontrivial on $H_1(\Sigma_{g,1};\bbZ)$, we have the equality
\[\ker(\Psi_N)=\ian= \ker(\Psi_N|_{\text{Aut}^{+}}:\mcgn\to\Sp).\] \\
Section 1 is divided into the following steps: \\

\begin{enumerate}[label=1.\Alph*.]
 \item\label{injective} The map $\al$ is injective.\\

\item\label{kernels} The Torelli group $\iap$ is contained in the Torelli group $\ian$.\\
 
\item\label{iso} The map $\bar{\al}$ is an isomorphism onto its image.\\
\end{enumerate}
The containment $N\subset\iap$ will follow easily from part \ref{iso}\\

\textbf{\ref{injective} Injectivity of $\al$.}
The map $\al$ is defined by the conjugation action of $\mcg$ on $N$. Thus, $\ker(\al)$ centralizes $N$. Since $N\cong \pi_1(\Sigma_g)$ has trivial center, $\ker(\al)\cap N=1$.\\ 

Let $\phi\in N\triangleleft\mcg$ be nontrivial. There is some $x\in \pp$ such that $\phi(x)\neq x$. That is, by Equation ($\dagger$), the element $\phi x\phi^{-1}x^{-1}$ of $\mcg$ is nontrivial. However, because both $N$ and $P$ are normal in $\mcg$ it follows that $\phi x\phi^{-1}x^{-1}\in N \cap P$. Therefore, the intersection $N\cap P\not=\emptyset$.\\

Likewise, if $\ker(\al)\not=1$, then there is a nontrivial element of $\ker{\al}\cap P$.\\

The two subgroups $\ker{(\al)}\cap P$ and $P\cap N$ are commuting subgroups of $P$. Because $\ker(\al)\cap N=1$, the intersection $(\ker{(\al)}\cap P)\cap (P\cap N)$ is trivial. However, $x_1,x_2\in \pp$ commute if and only if $x_1=\omega^k$ and $x_2=\omega^m$ for some $\omega\in\pp$ (see, e.g. \cite[Sect.1.1.3]{primer}). For $\ker{(\al)}\cap P$ and $P\cap N$ to intersect trivially and also commute, it must be that $\ker{(\al)}=1$. Therefore $\al$ is injective.\\

\textbf{\ref{kernels} Containment of Torelli groups.} Let $\Psi_\Sigma, \Psi_N,$ and $\al$ be defined as in ($\ddagger$). The following theorem of Korkmaz relates the two homomorphisms $\Psi_\Sigma$ and $\Psi_N\circ\al$.
\begin{thm}\label{kor}(Korkmaz \cite[Thm.1]{korkmaz}). For $g\geq 3$, any group homomorphism $\phi:\mcg\to \text{Gl}_{2g}(\bbC)$ is either trivial or else conjugate to the standard representation $\Psi_\Sigma:\mcg\to\Sp$.
\end{thm}
 Two homomorphisms $\phi,\psi:G\to H$ are \emph{conjugate} if there exists an element $h\in H$ such that $h\phi h^{-1}=\psi(g)$ for all $g\in G$. Note that conjugate homomorphisms have the same kernel.\\

By Theorem \ref{kor}, the composition \[\Psi_N\circ\al:\mcg\to \Sppm\subset \text{GL}_{2g}(\bbC)\] is either trivial or conjugate to $\Psi_\Sigma$. Thus, the kernel of $\Psi_N\circ\al$ is either all of $\mcg$ or exactly $\iap$. In either case, $\al(\iap)\subset\ker(\Psi_N)=\ian$. Using the injectivity of $\al$ to simplify notation, $\iap\subset\ian$.\\

\textbf{\ref{iso} The map $\bar{\al}$ is an isomorphism.} Using the fact that $\ker(\Psi_\Sigma)\subset\ker(\Psi_N)$, there is a well-defined homomorphism $\bar{\al}:\Sp\to\Sppm$ which makes the diagram $(\ddagger)$ commute.
Note that $\mcg$ contains torsion elements but $\ian$ is torsion-free (see \cite[Sect.2 pp.101)]{hain}). Therefore, $\al(\mcg)\not\subset\ian$, and $\Psi_N\circ\al\not=1$. The commutativity of ($\ddagger$) implies $\bar{\al}\circ \Psi_\Sigma\not=1$. Again applying Theorem \ref{kor}, the image of $\bar{\al}$ must be conjugate to $\Sp$. Therefore $\bar{\al}$ is an isomorphism onto its image.

Because $N\subset \ker(\Psi_N)$ and because the diagram $(\dagger)$ commutes, it follows that $N\subset \ker(\bar{\al}\circ\Psi_\Sigma)$. However, $\ker(\bar{\al})=1$ implies $N\subset\ker(\Psi_\Sigma)$. That is \[N\subset\iap\subset\ian.\]

\section*{2. The second term of the Johnson filtration.}
\addtocounter{section}{1}
\setcounter{subsection}{0}
In this section we will ``push" $P$ and $N$ deeper into the second term of Johnson filtration. We will prove that $N\subset P\cdot \iapt$. To that end, we will prove $N\cdot\iap/\iap=P\cdot\iap/\iap$ by using the Johnson homomorphism and the representation theory of $\Spq$.

\subsection{Johnson filtration of $\mcg$ and lower central series of $P$} In this subsection, we will consider the quotient $P\cdot\iap/\iap\cong P/(P\cap\iap)$. We will prove that $P\cap\iap=[P,P]$. Moreover, we will establish the following general fact:
\begin{prop}\label{mainlemma}$P\cap \iap_k=\gamma_k(P)$ for all $k\geq 1$.
\end{prop}

To condense notation, let $\gamma_k:=\gamma_k(\pp)$ be the $k$th term of the lower central series. Notice that:
\begin{equation} \label{rewrite}
 \pp\cap \iapk 
=\{x\in\pp\,|\, xyx^{-1}y^{-1}\in\gamma_{k+1}\text{ for all } y\in \pp\}.\end{equation}
That is, $x\in\pp\cap \iapk$ if and only if the left coset  $x\gamma_{k+1}$ is contained in the center $Z(\gamma_1/\gamma_{k+1})$. Thus, Proposition \ref{mainlemma} is equivalent to showing that:
\[Z(\gamma_1/\gamma_{k+1})=\gamma_k/\gamma_{k+1}\hspace{.3cm}\text{for all }k\geq 1.\] We will demonstrate this equality by establishing two containments. The containment $Z(\gamma_1/\gamma_{k+1})\supset \gamma_k/\gamma_{k+1}$ follows from the definition of the lower central series. The opposite containment relies on an analysis of the center of the Lie algebra associated to the lower central series of $\pp$, described below.\\

Associated to the lower central series of any group is a graded Lie. (See e.g. work of Lazard in \cite{lcs}, or Labute  in \cite{Labute1970}. Mal'cev is credited with first using this nilpotent filtration to study groups in \cite{malcev}.) Specifically for $G=\pp$, define 
\[\Lambda_i:=\gamma_i(\pp)/\gamma_{i+1}(\pp)\,\,\,\,\text{ for } i\geq 1. \]  
Each $\Lambda_i$ is a $\bbZ$-module. The sum \[\Lambda:=\bigoplus_i\Lambda_i\] can be given the structure of a graded Lie algebra over $\bbZ$ as follows. Let $(\,,\,)$ be the commutator in $\pp$.  The Lie bracket $[\, ,\,]$ is induced on $\Lambda$ by the commutator. That is, for $x\in \Lambda_k$, $y\in \Lambda_j$, and $\bar{x},\bar{y}$ lifts of $x, y$ respectively to $\pp$, we define
\[[x,y]:=(\bar{x},\bar{y})\gamma_{k+j+1}\in\Lambda_{k+j}. \] Despite the fact that $\pp$ is written as a multiplicative group, $\Lambda_i$ is written additively as a $\bbZ$-vector space. In particular, the left coset $1\gamma_{k+1}$ is $0$ as an element of $\Lambda_k$.\\

To prove Proposition \ref{mainlemma} it remains to show $Z(\gamma_1/\gamma_{k+1})\subset\gamma_k/\gamma_{k+1}$.
 We will divide the proof into two main steps as follows:
 \begin{enumerate}[label=\ref{mainlemma}.\Alph*.]
\item\label{trivialcenter} $\Lambda$ has trivial center $\implies$ $Z(\pi_1/\gamma_{k+1}(\pi_1))\subset \gamma_{k}(\pi_1)/\gamma_{k+1}(\pi_1)$.\\
\item\label{uae} The universal enveloping algebra $U(\Lambda)\cong \mathcal{A}_{2g}/\mathfrak{R}$ where $\mathcal{A}_{2g}$ is the free associative algebra on $2g$ indeterminates, and $\mathfrak{R}$ is the ideal generated by $\sum_{i=1}^g(a_i\otimes b_i-b_i\otimes a_i)$.\end{enumerate}

 To conclude, we will check that $U(\Lambda)$ has trivial center. Because all relations in $\Lambda$ must hold in its universal enveloping algebra, if $U(\Lambda)$ has trivial center, then so does $\Lambda$, and Proposition \ref{mainlemma} follows.

\begin{proof}[Proof of Proposition \ref{trivialcenter}\nopunct]
Let $x\in Z(\gamma_1/\gamma_{k+1})$. Suppose for the sake of contradiction $x\notin\Lambda_k$. We will show that $x\in Z(\Lambda)$.\\
 
There is some smallest $i\geq 1$ such that $x\in\gamma_{k-i}/\gamma_{k+1}$. Let $y\in \Lambda_j$. In order to show that $x$ is central, because the Lie bracket is bilinear, it suffices to check that $[x,y]=0$ for any $y$ and any $j$. That is, the commutator $(\bar{x},\bar{y})\in\gamma_{k-i+j+1}$.\\
  
  Consider the two cases: either $1\leq j\leq i$ or $1\leq i<j$.\\
  
   First, let $1\leq j\leq i$.
 Notice, for any $\bar{y}\in \pp$, the commutator $(\bar{x},\bar{y})\in \gamma_{k+1}$ because $x\in Z(\gamma_1/\gamma_{k+1})$. Since $j\leq i$, it follows that $\gamma_{k+1}\subset{\gamma_{k-i+j+1}}$. Thus, \[(\bar{x},\bar{y})\in\gamma_{k+1}\subset{\gamma_{k-i+j+1}}.\] Therefore, if $y\in \Lambda_{j}$ for $j\leq i$ then \[[x,y]=1\gamma_{k-i+j+1}=0\in \Lambda_{k-i+j}.\]\\
 \item[] Otherwise, let $1\leq i<j$. We will prove $(\bar{x},\bar{y})\in\gamma_{k-i+j+1}$ by induction on $j$. Suppose first $j=2$ (forcing $i=1$). Without loss of generality we may assume $\bar{y}=(\bar{a},\bar{b})$. The Jacobi identity provides
 \[[x,y]=[x,[a,b]]=-[b,[x,a]]-[a,[b,x]].\]
 However, $(\bar{x},\bar{a}),(\bar{b},\bar{x})\in\gamma_{k+1}$ because $x\in Z(\gamma_1/\gamma_{k+1})$. This implies \[-(b,(x,a))-(a,(b,x))\in \gamma_{k+2}=\gamma_{k-1+2+1}=\gamma_{k-i+j+1}.\] Therefore, $[x,y]=0$.
\\

To complete the induction, let $M<k$. Assume if $i\leq j\leq M$, then $(\bar{x},\bar{y})\in\gamma_{k-i+j+1}$ for all $y\in\Lambda_j$. Suppose $y\in \Lambda_{M+1}$. Without loss of generality we may assume that $y$ is an ($M+1$)-fold commutator, i.e. $y=(\bar{a},\bar{b})\gamma_{M+2}$ for some $a\in\Lambda_1$ and $b\in \Lambda_{M}$. By assumption, $(\bar{x},\bar{a})\in\gamma_{k+1}$, which implies \[(\bar{b},(\bar{x},\bar{a}))\in \gamma_{k+1+M+1}\subset\gamma_{k-i+M+1}.\] By the inductive hypothesis, $(\bar{b},\bar{x})\in\gamma_{k-i+M+1}$. Therefore, \[(\bar{x},(\bar{a},\bar{b}))=-(\bar{b},(\bar{x},\bar{a}))-(\bar{a},(\bar{b},\bar{x}))\in \gamma_{k-1+M+1}\subset \gamma_{k-i+M+1}\] 
implying that $[x,y]=0$ for all $y\in \Lambda$.

 Therefore, $x\neq 0$ is central in $\Lambda$. This proves the implication
\begin{equation} \label{centerlessimplies} Z(\Lambda)=0\Rightarrow Z(\gamma_1/\gamma_{k+1}) \subset\gamma_{k}/\gamma_{k+1}\end{equation}  \end{proof} 

\begin{proof}[Proof of Proposition \ref{uae}\nopunct] The following theorem of Labute shows that the graded Lie algebra $\Lambda$ is a quotient of the free Lie algebra on $2g$ generators by a principal ideal. \begin{thm} (Labute \cite{Labute1970}).
Let $\mathcal{L}_{2g}$ be the free Lie algebra on $2g$ generators (denoted $a_1,b_1,\dots,a_g,b_g$). Let $\mathcal{R}$ be the ideal generated by $\sum_i[a_i,b_i]$. Then $\Lambda\cong\mathcal{L}_{2g}/\mathcal{R}$.
\end{thm}

 Let $T(\Lambda)$ be the tensor algebra on the vector space underlying $\Lambda$. Let $U(\Lambda)$ be the universal enveloping algebra. Define $U(\Lambda)$ as \[U(\Lambda):= T(\Lambda)/\langle a\otimes b-b\otimes a -[a,b]\rangle.\] Let $\mathcal{A}_{2g}$ be the free associative algebra on $2g$ indeterminates. Let $\mathfrak{R}$ be the ideal in $\mathcal{A}_{2g}$ generated by $\sum_{i=1}^g(a_i\otimes b_i-b_i\otimes a_i)$.  We will prove below that the universal enveloping algebra $U(\Lambda)$ is isomorphic to $\mathcal{A}/\mathfrak{R}$.\\
 
The analogous fact for free groups, $U(\mathcal{L}_n)\cong\mathcal{A}_n$, was established by Magnus-Karrass-Solitar, see \cite[pp.347 ex.5]{MKS}. Let $U:\mathcal{L}\to \mathcal{Assoc}$ be the functor from the category of Lie algebras to the category of associative algebras that takes a Lie algebra to its universal enveloping algebra. 

Let $G: \mathcal{A}\to \mathcal{L}$ be the functor from the category of associative algebras to the category of Lie algebras, that induces the Lie bracket by the commutator in the associative algebra. $U$ is left-adjoint to $G$.

Define an injection
\[\phi:\mathcal{L}_1\to\mathcal{L}_{2g}\]
via\[\phi(1):=\sum_i[a_i,b_i].\]

Notice that: \[U(\mathcal{L}_1\stackrel{\phi}{\longrightarrow}\mathcal{L}_{2g})=\mathcal{A}_1\stackrel{U(\phi)}{\longrightarrow}\mathcal{A}_{2g}.\]

The map $U(\phi)$ is the injective map defined by $U(\phi)(1):= \sum(a_i\otimes b_i-b_i\otimes a_i)$. Note, coker$(\phi)\cong\Lambda$ and coker$(U(\phi))\cong\mathcal{A}_{2g}/\mathfrak{R}$. Since $U$ is left-adjoint to $G$ it preserves cokernels, meaning $U(\text{coker}(\phi))\cong \text{coker}(U(\phi))$. Therefore $U(\Lambda)\cong \mathcal{A}_{2g}/\mathfrak{R}$.\end{proof}

 In order to show that $\Lambda$ is centerless, it is sufficient to show that the universal enveloping algebra, $U(\Lambda)$ is centerless.
A computation of Crawley-Boevey-Etingof-Ginzburg in  \cite[Thm.8.4.1(ii)]{ginzburg} shows that the Hochschild cohomology $HH^{0}(\mathcal{A}_{2g}/\mathfrak{R})\cong\bbZ$. For an associative algebra $\mathcal{A}$, the center $Z(\mathcal{A})=HH^{0}(\mathcal{A})$ (see e.g. \cite[Sect.9.1.1]{weibel}). That is, only $\bbZ$ is central in the associative $\bbZ$-algebra $\mathcal{A}/\mathfrak{R}=U(\Lambda)$. All relations in $\Lambda$ must be preserved in $U(\Lambda)$. Thus, $Z(\Lambda)=0$.\\

Because $Z(\Lambda)=0$, it follows from Equation (\ref{centerlessimplies}) that  $Z(\gamma_1/\gamma_k)\subset\gamma_k/\gamma_{k+1}$. Thus, $Z(\gamma_1/\gamma_k)=\gamma_k/\gamma_{k+1}$. Then, by Equation (\ref{rewrite}) it follows that $P\cap\ia_k=\gamma_k(P).$

In particular, we have shown that $P\cap\iapt=\gamma_2(P)$, and equivalently $N\cap\iant=\gamma_2(N)$.

\subsection{Abelizanization of $\iap$}
We have already established that: \begin{eqnarray*}N\cap[N,N]&\subseteq& N\cap[\iap,\iap]\\ &\subseteq& N\cap[\ian,\ian]\\ &\subseteq& N\cap\iant\\&=&[N,N].\end{eqnarray*} The first containment follows from $N\subset\iap$ (Section 1). The second containment follows from $\iap\subset\ian$ (Section 1). Johnson's work showing that $\ian/\iant$ is abelian implies the third containment. The final equality is a consequence of Proposition \ref{mainlemma}. To conclude that $N\cap\iapt=[N,N]$ it suffices to check that $\iapt\subset\iant$. To establish this containment, we need to study the Johnson homomorphism and Johnson filtration. \\

 Let $\boundary$ be a compact surface of genus $g$ with one boundary component. Let $\mcgb$ be the group of isotopy classes of orientation preserving homeomorphisms of $\boundary$ fixing the boundary pointwise. Define the Torelli group for a surface with boundary as \[\iab:=\ker(\Psi:\mcgb\to\Sp)\] where $\Psi$ is the standard symplectic representation. For emphasis, we will sometimes distinguish as $\iast$ the Torelli group for a once marked surface. Unless otherwise specified $\ia=\iast$.
 Let $x\in H_1(\boundary;\bbZ)$, let $\bar{x}\in\po(\boundary)/\gamma_3(\po(\boundary))$ be a representative of $x$. Let $\phi\in\iab$. The Johnson homomorphism for $\ia_g^1$ is
 \[\uptau_g^1:\iab\to \text{Hom}(H_1(\boundary;\bbZ),\gamma_2(\po(\boundary))/\gamma_3(\po(\boundary))\]
 given by  
\[\uptau_g^1(\phi)(x)=\phi(\bar{x})\bar{x}^{-1}.\]
Many properties of $\iast$ follow directly from the properties of $\iab$. In a series of papers, Johnson established several important results summarized in the following theorem:
\begin{thm}[(Johnson)]\label{johnson} Let the notation be as above. For $g\geq 3$, the following hold: \begin{enumerate}[label=\Alph*.]
\item $Im(\uptau_g^1)\cong\Lambda^3H_1(\boundary;\bbZ)$  \cite{johnsonsurvey}. 
\item $H_1(\iab;\bbQ)\cong\Lambda^3H_1(\boundary;\bbQ)$ \cite{johnsonstructure1, johnsonstructure3, johnsonsurvey}.
\item  $\ker(\uptau_g^1)=(\iab)_2$ see, e.g. \cite[Th.6.18]{primer}.
\item  $\uptau_g^1:\iab/(\iab)_2\to\Lambda^3H_1(\boundary;\bbZ)$ is an $\Sp$-equivariant isomorphism, see, e.g. \cite[Eq.6.1]{primer}.
\item The quotient $\iab/(\iab)_2$ is the universal torsion-free abelian quotient of $\iab$ see \cite{johnsonsurvey, abelianquotient} or e.g. \cite[Sect.6.6.3]{primer}).
\end{enumerate}
 \end{thm}

To condense notation, let \begin{eqnarray*}
H_\bbZ&:=&H_1(\boundary;\bbZ),\\
H_\bbQ&:=&H_1(\boundary;\bbQ),\\
\pi_1&:=&\pi_1(\Sigma_{g}).
\end{eqnarray*}  
We can define the Johnson homomorphism for $\iast$ as follows.
 Let $x\in H_\bbZ$, let $\bar{x}$ a representative of $x$ in $\po$, and $\phi\in\ia$. Define \[\uptau:\iast\to\text{Hom}(H_\bbZ,\gamma_2(\pi_1)/\gamma_3(\pi_1))\] by
 \[\uptau(\phi)(x)=\phi(\bar{x})\bar{x}^{-1}.\]

The map $\uptau$ is well-defined by Proposition \ref{mainlemma}. 
 Let $T_\partial$ be the Dehn twist about the boundary curve of $\boundary$. The fact that $T_\partial\in\ker(\uptau_g^1)$ 
  implies that $\uptau_g^1: \iab\to \Lambda^3H_\bbZ$ factors through $\iast$ (see \cite{johnsonsurvey}). As such, Theorem \ref{johnson} A-E holds for $\iast$ and $\uptau$.\\

By showing that $\iapt\subset\iant$, we will conclude that: \[[N,N]=N\cap[\iap,\iap]=N\cap\iapt=N\cap\iant=[N,N].\] 

\begin{rem} The quotient $\iap/\iapt$ differs from the universal abelian quotient of $\iap$ only in torsion. That is,
\begin{eqnarray*}\iap/\iapt\otimes \bbQ&\cong& H_1(\iap;\bbQ)\\ &\cong& \iap/[\iap,\iap]\otimes \bbQ.\end{eqnarray*} Therefore, $[\iap,\iap]\subset \iapt$ and the quotient $\iapt/[\iap,\iap]$ is isomorphic to the torsion subgroup of $\iap/[\iap,\iap]$.
\end{rem}

To see that $\iapt\subset\iant$, we will consider the difference between $H_1(\ia;\bbZ)$ and $H_1(\ia,\bbQ)$.
 
The abelianization, $H_1(\iast;\bbZ)$, can be computed using techniques employed by Johnson in \cite{johnsonstructure3} to compute $H_1(\iab;\bbZ)$. We could not find this exact computation in the literature, so we give it below.\\

\begin{prop}\label{computeH1}
$H_1(\iast;\bbZ)\cong \Lambda^3 H_\bbZ\oplus \mathcal{B}_2/\langle a\rangle$ where $\mathcal{B}_2/\langle a\rangle$ is 2-torsion (defined explicitly below).
\end{prop}

A \emph{boolean polynomial} is a polynomial with coefficients in $\bbZ/2\bbZ$. Define $\mathcal{B}_i$ to be the group of boolean polynomials $p$  on $2g$ indeterminates with $\deg(p)\leq i$. Building on the work of Birman-Craggs in \cite{birmancraggs}, Johnson constructed in \cite[Th.6]{johnsonbirman} (see also e.g. \cite[Th.6.19]{primer}) a surjective homomorphism \[\sigma:H_1(\iab;\bbZ)\to \mathcal{B}_3\] such that the torsion of $H_1(\iab;\bbZ)$ is captured by $\mathcal{B}_2$. In addition, Johnson constructed the surjective $\Sp$-equivariant homomorphism
\[q:\mathcal{B}_3\to \Lambda^3H_\bbZ\otimes \bbZ/2\bbZ;\] for details see \cite{johnsonstructure3}[Prop.4]. Explicitly, Johnson computed $H_1(\iab;\bbZ)\cong \Lambda^3H\oplus \mathcal{B}_2$ using these two homomorphisms and pullback diagrams of groups. A \emph{pullback diagram} for the group homomorphisms $\psi_1:A\to C$ and  $\psi_2:B\to C$ is
\begin{equation}\label{pulldi}\begin{tikzcd}
{}
& D \arrow{dr}{\phi_1} \arrow{dl}[swap]{\phi_2} \\
A \arrow[swap]{dr}{\psi_1} & & 
B \arrow{dl}{\psi_2} \\
& C
\end{tikzcd} 
\end{equation}
 a commutative square (\ref{pulldi}) that is terminal among all such squares. That is, the pullback $(D,\phi_1,\phi_2)$ is universal with respect to the diagram (\ref{pulldi}). For a diagram of groups, the pullback is \[D\cong\{(a,b)\in A\times B\,|\, \psi_1(a)=\psi_2(b)\}.\] $D$ is unique up to canonical isomorphism.\\

Diagram D1 (below) is a pullback diagram, from which Johnson in \cite{johnsonstructure3} concludes that $H_1(\iab;\bbZ)\cong \Lambda^3H\oplus \mathcal{B}_2$.
\begin{center}
\begin{tikzcd}
{}
& H_1(\iab,\bbZ) \arrow{dr}{\uptau} \arrow{dl}[swap]{\sigma} \\
\mathcal{B}_3 \arrow[swap]{dr}{q} & & 
\Lambda^3H_\bbZ\arrow{dl}{\otimes \bbZ/2\bbZ} \\
& \Lambda^3H_\bbZ\otimes \bbZ/2\bbZ
\end{tikzcd}\\
\textbf{D1:} Pullback diagram used to compute $H_1(\iab;\bbZ)\cong \Lambda^3H_\bbZ\oplus B_2$.\\

\end{center}

Let $T_\partial$ be the Dehn-twist about the boundary component in $\iab$.  In order to compute $H_1(\iast;\bbZ)$ note that: \[\iast\cong\iab/\langle T_\partial\rangle .\]   
Define $a\in\mathcal{B}_2$ as \[a:=\sum_i a_ib_i.\] In \cite{johnsonstructure3}, Johnson computes \[\sigma(T_\partial)=a.\]

\begin{proof}[Proof of Proposition \ref{computeH1}]
We will use two additional pullback diagrams to compute  $H_1(\iast;\bbZ)\cong \Lambda^3 H_\bbZ\oplus \mathcal{B}_2/\langle a\rangle$. Define the quotient map
\[f:\ia_g^1\to \ia_g^1/\langle T_\partial\rangle\cong \iast.\]
 The inverse image of the commutator subgroup of $\iast$ is 
\[f^{-1}([\iast,\iast])=f^{-1}(([\iab,\iab]\cdot\langle T_\partial\rangle)/\langle T_\partial\rangle )=[\iab,\iab]\cdot \langle T_\partial\rangle.\]
Define the quotient map  \[g:\iab\to \frac{(\iab/\langle T_\partial \rangle)}{(([\iab,\iab]\cdot \langle T_\partial\rangle)/\langle T_\partial\rangle)}\cong H_1(\iast;\bbZ) .\]
The kernel of $g$ is exactly $[\iab,\iab]\cdot \langle T_\partial \rangle $. Thus, there is an isomorphism 
\[g:\iab/([\iab,\iab]\cdot \langle T_\partial \rangle )\to H_1(\iast;\bbZ).\]
 
Notice that: \[H_1(\iab;\bbZ)\twoheadrightarrow\frac{(\iab/[\iab,\iab])}{([\iab,\iab]\cdot\langle T_\partial\rangle/[\iab,\iab] )}\cong \frac{\iab}{([\iab,\iab]\cdot \langle T_\partial\rangle)}.\] Therefore we have a map  $H_1(\iab;\bbZ)\to H_1(\iast;\bbZ)$ with kernel $[\iab,\iab]\cdot \langle T_\partial \rangle /[\iab,\iab]$. Construct the following pullback diagrams D2 and D3:\\

\begin{center}
\begin{tikzcd}
{}
& (\langle T_\partial\rangle\cdot[\iab,\iab])/[\iab,\iab]  \arrow{dr}{\uptau} \arrow{dl}[swap]{\sigma} \\
\langle a \rangle \arrow[swap]{dr}{} & & 
1\arrow{dl}{} \\
& 1
\end{tikzcd}\\
\textbf{D2:} Because $\sigma$ is an isomorphism, this is a pullback diagram.\\
\end{center}

Taking a quotient of D1 by D2 results in the following pullback diagram D3:\\

\begin{center}
\begin{tikzcd}
{}
& H_1(\iast,\bbZ) \arrow{dr}{\uptau} \arrow{dl}[swap]{\sigma} \\
B_3/\langle a\rangle \arrow[swap]{dr}{q} & & 
\Lambda^3H_\bbZ\arrow{dl}{\otimes\bbZ/2\bbZ} \\
& \Lambda^3H_\bbZ\otimes \bbZ/2\bbZ
\end{tikzcd}\\
\textbf{D3:} The pullback diagram quotient of D1 by D2. Diagram D3 can be used to compute $H_1(\iast;\bbZ)\cong\Lambda^3H\oplus \mathcal{B}_2/\langle a \rangle$.\\
\end{center}

 Johnson showed that D1 is a pullback diagram in \cite{johnsonstructure3}. D2 is a pullback diagram because $\langle a \rangle \cong (\langle T_\partial\rangle[\iab,\iab])/[\iab,\iab]\cong \bbZ/2\bbZ$. Since D3 is a quotient of two pullback diagrams and one terminal homomorphism of D1 is surjective, it follows that D3 is also a pullback diagram. Therefore, $H_1(\iast;\bbZ)\cong \Lambda^3H\oplus \mathcal{B}_2/\langle a\rangle$.\end{proof}
 
\subsection{Intersection of $N$ with $\iapt$}\label{Niniapt}
 The homomorphism $\al$, as defined in Section 1, gives the injection $\mcg\hookrightarrow\text{Aut}^\pm(N)$. From Section 1, the containment $\iap\subset\ian$ implies $[\iap,\iap]\subset[\ian,\ian]$. Define $\bar{\uptau}_P$ (respectively, $\bar{\uptau}_N$) as the quotient map \[\bar{\uptau}_P:\iap\to \iap/[\iap,\iap]\cong\Lambda^3H_\bbZ\oplus\mathcal{B}_2/\langle a\rangle.\]  Since $[\iap,\iap]\subseteq[\ian,\ian]$ it follows that $\ker(\bar{\uptau}_P)\subset\ker(\bar{\uptau}_N)$. Thus, we can define a homomorphism $\tilde{\al}$ so that the right hand square of (\ref{torsiondia}) commutes.
  \begin{equation}\label{torsiondia}
   \begin{tikzcd}
 1 \arrow{r}{}  & \left[\iap,\iap\right] \arrow{d}{\al} \arrow{r}{} & \iap\arrow{r}{\bar{\uptau}_p}\arrow{d}{\al} & \Lambda^3H_\bbZ\oplus B_2/\langle a\rangle \arrow{d}{\tilde{\al}}\arrow{r} & 1\\
 1 \arrow{r}{}  & \left[\ian,\ian\right] \arrow{r}{} & \ian\arrow{r}{\bar{\uptau}_N} & \Lambda^3H_\bbZ\oplus B_2/\langle a\rangle \arrow{r} & 1
 \end{tikzcd}
   \end{equation}

 The fact that $\tilde{\al}$ must map torsion to torsion implies that $\tilde{\al}(\mathcal{B}_2/\langle a \rangle )\subset\mathcal{B}_2/\langle a \rangle$. Thus,
 \[\bar{\uptau}_N(\al(\iapt))=\tilde{\al}(\bar{\uptau}_p(\iapt))\subset \mathcal{B}_2/\langle a \rangle.\]
 This containment implies \[\al(\iapt)\subset \bar{\uptau}_N^{-1}(\mathcal{B}_2/\langle a \rangle )=\iant.\]
 Therefore $\iapt\subset \iant$.\\

The containment $\iapt\subset\iant$ allows us to deduce the following:
\[[N,N]\subseteq N\cap [\iap,\iap]\subseteq N\cap\iapt\subseteq N\cap \iant= [N,N].\]

Therefore, $N\cap\iapt=\gamma_2(N)$.

\subsection{$\Spq$ representation}\label{Niniap} In this subsection, we will use the representation theory of $\Spq$ to show that $N\cdot\iapt=P\cdot\iapt$.\\

$\mcg$ acts on $\iap$ via conjugation. The kernel of $\uptau$ is exactly the set of elements that act trivially on $\pi_1/\gamma_3(\pi_1)$, i.e. $\ker(\uptau)=\iapt$. The quotient $\iap/\iapt$ is the universal torsion-free abelian quotient of $\iap$. Thus, the conjugation action of $\iap$ on $\iap/\iapt$ is trivial. Therefore, we have a well-defined action of $\mcg/\ia\cong\Sp$ on $\ia/\ia_2$. Similarly, $\Sp$ has a canonical action on $\Lambda^3 H_\bbZ$. The isomorphism \[\uptau:\iap/\iapt\to \Lambda^3 H_\bbZ\] is $\Sp$-equivariant. \\ To prove that $N\cdot \iapt/\iapt=P\cdot \iapt/\iapt$ we will establish the following bijective correspondence:\\ 
\begin{center}
 $\Big\{$ \begin{tabular}[t]{m{3.7cm}}
  $\Spq$-irreps in $\Lambda^3H_\bbQ$ 
 \end{tabular}$\Big\}\longleftrightarrow\Big\{$ \begin{tabular}[t]{m{4.3cm}}
 $\Sp$-invariant $\bbZ$-module direct summands in $\ia/\ia_2$
 \end{tabular}$\Big\}.$ \end{center}
 \vspace{.25cm}
  We will check that there is exactly one $\Spq$-invariant, dimension $2g$ subspace of $\Lambda^3H_\bbQ$. To conclude, we will show that both $N\iapt/\iapt$ and $P\iapt/\iapt$ are rank $2g$ direct summands of $\iap/\iapt$ invariant under the action of $\Sp$.

\begin{lem}[(Bijective correspondence)]\label{spirrep} 
There is a bijective correspondence between $\Spq$-invariant dimension-$m$ $\bbQ$-vector subspaces of $\Lambda^3H_\bbQ$ and $\Sp$-invariant rank $m$ $\bbZ$-module direct summands of $\Lambda^3H_\bbZ$.
\end{lem}

\begin{proof}[Proof of Lemma \ref{spirrep}] 
Define the map \[f: \{\Sp\text{-invariant direct summands of }\Lambda^3H_\bbZ\}\to \{\text{subspaces of }\Lambda^3H_\bbQ\}\] via \[f(V)= V\otimes \bbQ \,.\] To establish the bijective correspondence, we need to check that the image of $f$ lies in $\Spq$-invariant subspaces of $\Lambda^3H_\bbQ$. 

Fix a basis of $\bbQ^{2g}$ so that $\Spq<\gl$ is the subgroup that fixes the symplectic form $\left(\begin{array}{c|c}
0 & I_{g\times g}\\
\hline
-I_{g\times g}& 0
\end{array}\right)$. The group $\Spq$ is generated by matrices of the following forms, where $\lambda$ varies in $\bbQ$, and $e_{ij}$ is the $g\times g$ matrix with 1 in the $i,j$ entry and $0$ elsewhere (see e.g. \cite[Sect.2.2]{Spgen}):

\begin{equation}\label{spgen}
\left(\begin{array}{l|r}
I_{g\times g}& \lambda e_{ii}\\
\hline
& I_{g\times g}
\end{array}\right),\,\, \left(\begin{array}{l|r}
I_{g\times g}& \\
\hline
\lambda e_{ii} & I_{g\times g}
\end{array}\right),\,\, \left(\begin{array}{c|c}
I_{g\times g}& \\
\hline
\lambda( e_{ij}+e_{ji}) & I_{g\times g}
\end{array}\right), 
\end{equation}

\begin{equation*} 
\left(\begin{array}{c|c}
I_{g\times g}& \lambda( e_{ij}+e_{ji})\\
\hline
 & I_{g\times g}
\end{array}\right),\,\, \left(\begin{array}{l|r}
I_{g\times g} +\lambda e_{ij}\\
\hline
 & I_{g\times g}-\lambda e_{ji}
\end{array}\right).\end{equation*}

\vspace{.45cm}

Let $V$ be an $\Sp$-invariant direct summand of $\Lambda^3H_\bbZ$, and let $v\in V$. Let $A$ be any of the generators of $\Spq$ given in (\ref{spgen}) and let $A_\bbZ$ be the matrix $A$ with $\lambda=1$. Notice that $A_\bbZ\in\Sp$ and $A=\lambda A_\bbZ-(\lambda-1)I_{2g\times 2g}$. Therefore, for any $q\in \bbQ$: \[Aqv=qAv=q((\lambda)(A_\bbZ v)-\lambda v+ v)\, .\] Since $V$ is an $\Sp$-invariant direct summand, $q((\lambda)(A_\bbZ v)-\lambda v+ v) \in V\otimes \bbQ$.
Therefore $V\otimes \bbQ$ is an $\Spq$-invariant subspace.\\

Let $W$ be an $\Spq$-invariant subspace of $\Lambda^3H_\bbQ$. Let $W_\bbZ$ be the $\bbZ$-module consisting of all integral points of $W$. Define the map \[g:  \{\Spq \text{-invariant subspaces of }\Lambda^3H_\bbQ\}\to \{\Sp\text{-invariant direct summands of }\Lambda^3H_\bbZ\}\] via \[g(W)= W_\bbZ.\] 

 The composition $f\circ g$ is the identity because $W_\bbZ\otimes \bbQ=W$.\\
 
 On the other hand, consider $v\in g\circ f(V)=(V\otimes \bbQ)_\bbZ$. Decompose $\Lambda^3H_\bbZ=V\oplus V^\perp$. If $v\notin V$ then the projection of $v$ onto $V^\perp\not=0$. Let $p_\perp(v)$ be the projection onto $V^\perp$. Because $v\in V\otimes\bbQ$, it follows that $nv\in V$ for some large enough $n\in\bbZ$. However, that implies $p_\perp(nv)=0$, or equivalently $n(p_\perp(v))=0$, a contradiction. \\

Therefore, $g$ is a bijection and the correspondence is established.
\end{proof}

The representation $\Lambda^3H_\bbQ$ decomposes as an $\Spq$-representation  in the following way (see, e.g. \cite[Sect.3]{FarbSp}): \[\Lambda^3H_\bbQ\cong H_\bbQ\oplus\Lambda^3H_\bbQ/H_\bbQ.\]  Note that $\dim_\bbQ (H_\bbQ)=2g$ and $\dim_\bbQ(\Lambda^3H_\bbQ/H_\bbQ)=\binom{2g}{3}-2g$. 
Thus, for genus $g\geq 3$, there is exactly one $\Spq$-invariant, dimension-$2g$ subspace of $\Lambda^3H_\bbQ$.\\

From Section \ref{Niniapt} we have \[N\cap\iapt=[N,N].\]
Thus, \[N\iapt/\iapt\cong N/(N\cap\iapt)\cong N/[N,N]\cong \bbZ^{2g}.\]
 Therefore, $N\iapt/\iapt$ is a $\bbZ$-module of rank $2g$. Likewise, $P\iapt/\iapt$ is a $\bbZ$-module of  rank $2g$.\\

To see that the submodule $P\iapt/\iapt$ is a direct summand of $\iap/\iapt$, it is sufficient to check that the generators of $P$ surject onto a partial basis of $\Lambda^3H_\bbZ$ under the Johnson homomorphism. A \emph{partial basis} is any set of linearly independent vectors that can be completed to a $\bbZ$-basis. \\
Consider a fixed generating set for $P$ and a corresponding basis for $H_\bbZ$, given by $\{a_1,b_1,\dots, a_g,b_g\}$. Then, $\uptau(a_i)=\theta\wedge a_i$ where $\theta=\sum_i a_i\wedge b_i$. For details of this computation, see Johnson's work in \cite{abelianquotient}. The image of the standard generators of $P$ gives a partial basis of $\Lambda^3H_\bbZ$. Therefore the image of $P$ is a direct summand in $\Lambda^3H_\bbZ$. \\
 
 It remains to be seen that $N\cdot \iapt/\iapt$ is a direct summand. Because $[N,N]\subset[\iap,\iap]\subset[\ian,\ian]$, the following diagram given by restrictions of quotient maps commutes:
  \begin{equation}\label{abelianN}
  \begin{tikzcd}
 N/[N,N] \arrow{r}{k}\arrow{rd}[swap]{j}  & \iap/\iapt\arrow{d}{}  \\
  &\ian/\iant
 \end{tikzcd}
 \end{equation}
 
 The image $j(N/[N,N])=N\iant/\iant\cong\bbZ^{2g}$ is a direct summand in $\ian/\iant$. Further, $k(N/[N,N])=N\iapt/\iapt\cong\bbZ^{2g}$.  
 
\begin{lem}\label{direct} 
Suppose that the diagram below commutes \begin{center}
  \begin{tikzcd}
  {}
  \bbZ^{2g} \arrow{dr}[swap]{L_2}\arrow{r}{L_1}& \bbZ^{2g}\oplus\bbZ^{n-2g}\arrow{d}{L_3} \\
  &\bbZ^{2g}\oplus\bbZ^{n-2g}
  \end{tikzcd}
  \end{center}
   and the maps $L_i$ are linear. 
 If $L_2(\bbZ^{2g})\cong\bbZ^{2g}$ is a direct summand in $\bbZ^{2g}\oplus\bbZ^{n-2g}$, then so is $L_1(\bbZ^{2g})$.
 \end{lem}
 
 \begin{proof}[Proof of Lemma \ref{direct}] Because $L_2(\bbZ^{2g})$ is a direct summand in $\bbZ^n$, there exists a retract $R:\bbZ^{N}\to \bbZ^{2g}$ of $L_2$ with $R\circ L_2=Id_{\bbZ^{2g}}$. Further, since $L_2=L_3\circ L_1$, the homomorphism $R\circ L_3: \bbZ^{N}\to \bbZ^{2g}$ is a retract of $L_1$. That is $R\circ L_3\circ L_1=Id_{\bbZ^{2g}}$. Consider $L_1\circ R\circ L_3:\bbZ^{N}\to \bbZ^{2g}$. Note that: \[(L_1\circ R\circ L_3)^2=L_1\circ( R\circ L_3\circ L_1)\circ R\circ L_3=L_1\circ( Id_{\bbZ^{2g}})\circ R\circ L_3=L_1\circ R\circ L_3.\] It follows that $L_1\circ R\circ L_3$ is a projection with image $L_1(\bbZ^{2g})$. Thus, $L_1(\bbZ^{2g})$ is a direct summand.
 \end{proof}

Applying Lemma \ref{direct} to commutative diagram (\ref{abelianN}), it follows that $k(N/[N,N])=N\iapt/\iapt\cong\bbZ^{2g}$ is a direct summand in $\iap/\iapt$.\\ 
 
 Because $N, P,$ and $\iapt$ are normal in $\mcg$, both of the above $\bbZ$-module direct summands are invariant under the action of $\Sp$.
   There is exactly one rank $2g$ direct summand $\bbZ$-submodule of $\iap/\iapt$. Thus, $N\iapt/\iapt= P\iapt/\iapt$. Equivalently, $N\iapt=P\iapt$.

\section*{3. Commutator containment: $[N,N]\subset[P,P]$.}
\addtocounter{section}{1}
\setcounter{subsection}{0}
From Section 2, we have the containment $N\subset P\iapt$. Furthermore, since $[N,N]\subset \iapt$, it is also true that $[N,N]\subset P\iapt$. In this section, we will use an inductive argument to confirm that $[N,N]\subset P\iapk$ for all $k$. Grossman's Property A Lemma (see Lemma \ref{gross}) implies that for any surface group $\pp$, if $q\in \text{Aut}(\pp)$ preserves conjugacy classes in $\pp$, then $q\in \pp$.  Using Grossman together with the conjugacy $p$-separability of surface groups, we will show that $\cap_k P\ia_k=P$. This will prove $[N,N]\subset P\cap\iapt=[P,P]$.\\

\noindent We have already established the following facts:
\begin{enumerate}[label=\roman*.]
\item\label{centra} The Johnson filtration is a central series \cite{basslub}.
\item\label{NinTorelli} $N\subset \iap$ (Sect. 1).
\item\label{Niat} $N\subset P\iapt$ (from Sect. 2).
\item\label{intersectionofN} $N\cap\ \iank =\gamma_k(N)$ (Prop \ref{mainlemma}).
\item\label{iatiniatn} $\iapt\subset\iant$ (from Sect. 2).

\item\label{normal} $[\iapt,N]\subset N$ (because $N$ is normal in $\mcg$).
\item $[\iant,N]\subset \ia_3(N)$ (because $N\subset \ia$ and the Johnson filtration is a central series).\\

\noindent We will establish an eighth fact:
\item\label{pulloutp} $[PG,N]\subset P[G,N]$ for any $G\triangleleft\mcg$ (below).\\

\end{enumerate}
To prove (viii), let $G\triangleleft\mcg$. Let $g\in G$, $p\in P,$ and $  n\in N$ be given. Then \[[pg,n]=pgng^{-1}p^{-1}n^{-1}=pgnp^{-1}n^{-1}pg^{-1}p^{-1}(pgp^{-1}n(pgp^{-1})^{-1}n^{-1}).\] However, $P$ normal in $\mcg$ implies that: \[pgnp^{-1}n^{-1}pg^{-1}p^{-1}\in P.\] Furthermore, because $G$ is normal in $\mcg$ it follows that: \[((pgp^{-1})n(pgp^{-1})^{-1}n^{-1})\in [G,N].\]  Therefore $[PG,N]\subset P[G,N]$ for any $G\triangleleft\mcg$. In particular, $[P\iapt,N]\subset P[\iapt,N]$.\\

With reference to the above list of facts,
\[[N,N]\stackrel{(\ref{Niat})}{\subset} [P\iapt,N]\stackrel{(\ref{pulloutp})}{\subset} P[\iapt,N]\stackrel{(\ref{iatiniatn})}{\subset} P[\iant,N]\stackrel{(\ref{centra}),\,(\ref{NinTorelli})}{\subset} P(\ia_3(N)\cap N)\stackrel{(\ref{intersectionofN})}{\subset} P\gamma_3(N).\] 
Therefore, $[N,N]=\gamma_2(N)\subset P\gamma_3(N)$.\\

We will induct on $m$ to check that $\gamma_2(N)\subset P\gamma_m(N)$ for all $m>0$. Let $M\in\bbN$ with $M\geq 3$. Suppose for all $m\leq M$ we have $\gamma_2(N)\subset P\gamma_m(N)$. 
It follows that: \[[N,N]=\gamma_2(N)\subset P\gamma_3(N)=P[\gamma_2(N),N]\subset P[P\gamma_{M}(N),N]\subset P^2[\gamma_M(N),N]\subset P\gamma_{M+1}(N).\] Therefore, $[N,N]\subset P\gamma_m(N)$ for all $m\geq 1$.

We will use a second inductive argument to show that $\gamma_k(N)\subset \iapk$ for all $k\geq 2$. For the base case, note that: \[ [N,N]\subset[\iap,\iap]\subset \iapt.\] Assume as inductive hypothesis that $\gamma_k(N)\subset\iap_k$ for all $k< K$. Then \[\gamma_K(N)=[\gamma_{K-1}(N),N]\subseteq[\ia_{K-1}(\Sigma),\iap]\subseteq \ia_K(\Sigma).\]
 The above containment implies that \[[N,N]=\gamma_2(N)\subseteq \cap_k P\gamma_k(N)\subseteq \cap_k P\iap_k.\]

In order to confirm that $[N,N]\subset P$, it remains to be shown that  $\cap_k P\iap_k=P$. We will use the following Lemma due to Grossman:
  \begin{lem}[(Grossman's Property A \cite{grossman})]\label{gross}
    Let $P$ be a surface group of genus $g\geq 1$. Let $q\in \text{Aut}(P)$. If $q$ preserves conjugacy classes in $P$, then $q\in P$. 
   \end{lem}
   
   To apply Lemma \ref{gross}, choose $q\in \cap P\ia_k$ and $x\in P$. Since $q\in P\ia_k$ for all $k\geq 1$, we can find $u_k\in P,$ and $i_k\in \iapk$ such that $q=u_ki_k$. However, because $i_k\in\ia_k$ it follows that $i_kxi_k^{-1}x^{-1}\in\gamma_{k+1}(P)$. This can be rewritten in terms of left cosets as \[i_kxi_k^{-1}\gamma_{k+1}(P)=x\gamma_{k+1}(P).\] 
    Conjugating by $u_k$ gives \[u_ki_kxi_k^{-1}u_k^{-1}\gamma_{k+1}(P)=u_kxu_k^{-1}\gamma_{k+1}(P).\] That is, $q xq^{-1}$ is conjugate to $x$ in $P/\gamma_{k+1}(P)$ for all $k\geq 1$.\\
    
   Finite $p$-groups are nilpotent. Furthermore, any homomorphism $\phi:P\to H$ where $H$ is $i$-step nilpotent factors through $P/\gamma_{i+1}(P)$. Thus, any homomorphism $\phi:P\to H$ where $H$ is a $p$-group factors through $P/\gamma_{k}(P)$ for some $k$. \\  
   Suppose $\phi:P\to H$ gives a homomorphism to some $p$-group $H$. Because $q x q^{-1}$ is conjugate to $x$ in $P/\gamma_{k}(P)$ for all $k\geq 1$, it must be that $\phi (q x q^{-1})$ is conjugate to $\phi(x)$ in $H$. Because $P$ is conjugacy $p$-separable (see \cite{paris}), $q x q^{-1}$ is conjugate to $x$ in $P$.   Applying Lemma \ref{gross}, it follows that $q\in P$. Therefore, $\displaystyle \cap_k P\iapk=P$.\\
   
   We have established for all $k\geq 1$ \[\gamma_2(N)\subset P\gamma_k(N)\subset P\iapk.\] That is, $\gamma_2(N)\subset \cap_kP\iapk=P$. From Section 2, we have the containment $\gamma_2(N)\subset\iapt$. Thus, $\gamma_2(N)\subset P\cap\iapt=\gamma_2(P)$. This concludes the first main goal in the proof of Theorem \ref{maintheorem}: \[\gamma_2(N)\subseteq\gamma_2(P)\subseteq\iapt\subseteq\iap\subseteq\ian.\]
   
   \section*{4. Characterizing $P$}
   \addtocounter{section}{1}
   \setcounter{subsection}{0}
   In this section we will characterize $P$ in terms of $\iap$ and $\gamma_2(P)$. In the proof, we will show that any $\phi\in \iap$ satisfying certain conditions must fix a filling set of curves up to conjugation. Then, we will apply the Alexander method to show that $\phi$ must be isotopic to the identity in $\mcgg$.
   
   \begin{prop}[(Characterization of $P$)]\label{mainprop} For $g\geq 3$,
   \[P(\Sigma_g)=\{x\in\iap\,|\, [x,\iap]\subset\gamma_2(P(\Sigma_g))\}.\]
   \end{prop}
   
   The proof of Proposition \ref{mainprop} was greatly simplified by Chen Lei. 
     \begin{proof}  Because $P\triangleleft\mcg$ and $P\subset\iap$ it follows that for any $p\in P$ \[[p,\iap]\subset (\iapt\cap P)=\gamma_2(P).\] Therefore, \[P\subseteq \{x\in\iap\,|\, [x,\iap]\subset\gamma_2(P)\}.\]
   
  For the opposite containment, let $\phi\in\{x\in\iap\,|\, [x,\iap]\subset\gamma_2(P)\}$.
  
Our goal is to apply the Alexander method by demonstrating that  $\phi x_i\phi^{-1}$ is isotopic to $x_i$ for a filling set $\{x_i\}$ of simple closed curves. This would force $\phi$ to be isotopic to the identity in $\mcgg$. That is, $\phi\in P$.\\

   Take any bounding pair map, $T_aT_b^{-1}$ where $a$ and $b$ are disjoint, homologous, non-isotopic simple closed curves. Because $a$ and $b$ are homologous, it follows that $T_aT_b^{-1}$ acts trivially on $H_1(\Sigma_{g,1})$. That is, $T_aT_b^{-1}\in \iap$.  By assumption, $\phi T_aT_b^{-1}\phi^{-1}(T_aT_b^{-1})^{-1}\in P$. Mapping into $\mcgg$ via the forgetful map, we obtain, $F(\phi T_aT_b^{-1}\phi^{-1}(T_aT_b^{-1})^{-1})= 1$. That is \[\phi T_aT_b^{-1}\phi^{-1}(T_aT_b^{-1})^{-1}=1 \text{ in }\mcgg.\] Therefore
  \begin{eqnarray*}
  \phi T_aT_b^{-1}\phi^{-1}(T_aT_b^{-1})^{-1}&=&1\\
   \phi T_{(a)}\phi^{-1} \phi T_{(b)}^{-1}\phi^{-1}(T_aT_b^{-1})^{-1}&=&1\\
  T_{\phi(a)}T_{\phi(b)}^{-1}(T_aT_b^{-1})^{-1}&=&1\\
   T_{\phi(a)}T_{\phi(b)}^{-1}&=& T_aT_b^{-1}.\\
  \end{eqnarray*}
  Bounding pair maps commute if and only if they have the same canonical reduction system. Thus, $ T_{\phi(a)}T_{\phi(b)}^{-1}$ and $ T_aT_b^{-1} $ have the same canonical reduction system, namely $\{a,b\}$. As such, the curves $\phi({a})$ and $\phi({b})$ are isotopic to $a$ and $b$, respectively in $\mcgg$.
  
  For any non-separating simple closed curve $c$ there is a bounding pair map $T_c T_{c'}$ where $c$ and $c$ are homologous, disjoint, and non-isotopic. It follows that $\phi{c}$ is isotopic in $\mcgg$ to $c$ for any non-separating simple closed curve $c$.\\ 
        
     In particular, for a filling set of simple closed curves, $\{x_1,\dots,x_k\}$, we have $\phi x_i\phi^{-1}$ is isotopic to $x_i$ for each $i$. By the Alexander method, the map $\phi$ must be trivial in $\text{Out}(\pi_1(\Sigma_g))$. That is, $\phi\in P$. Proposition \ref{mainprop} follows.  \end{proof} 
    
    \section*{5. Conclusion: $N=P$.}
    \addtocounter{section}{1}
    \setcounter{subsection}{0}
     To conclude the proof of Theorem \ref{maintheorem}, write both $P$ and $N$ in the form given by Proposition \ref{mainprop}: \[P=\{x\in \iap\,|\,[x,\iap]\subseteq \gamma_2(P)\}.\] \[N=\{x\in \ian\,|\,[x,\ian]\subseteq \gamma_2(N)\}.\] From section 3, $[N,N]\subset [P,P]$ implies that: \[N\subset\{x\in \ian\,|\,[x,\ian]\subseteq \gamma_2(P)\}.\]
          From Section 1, $\iap\subset\ian$ implies that: \[N\subset\{x\in \ian\,|\,[x,\iap]\subseteq \gamma_2(P)\}.\]
          From Section 1, $N\subset \iap$ implies that: \[N\subset\{x\in \iap\,|\,[x,\iap]\subseteq \gamma_2(P)\}.\]
          Thus \[N\subseteq P.\]
          
          Since $N$ is a subgroup of $P$ and is not free, the index of $N$ in $P$ is finite (see e.g. \cite[Th.1]{jaco}).
          
          We can determine the index of $N$ via the following formula (see e.g. \cite[Sect.2.2 Ex.22]{hatcher}):  \[[N:P]\cdot \chi (\Sigma)=\chi(\Gamma)\] where $\chi$ is Euler characteristic. 
          Therefore $[N:P]=1$ and $N=P$. We have established Theorem \ref{maintheorem}.\\
    
    The following example demonstrates that $N$ need not equal $P$ if we remove the condition of normality. 

    \begin{exm}\label{example}
    Let $\varphi\in\mcgg$. Construct the mapping torus $M_\varphi\cong (I\times \Sigma_g)/(1,x)\sim(0,\varphi(x))$. Note that $\pi_1(M_\varphi)\cong \pp\rtimes\bbZ$. Consider the exact sequence \[1\longrightarrow \pp\longrightarrow\mcg\stackrel{F}{\longrightarrow} \mcgg\longrightarrow 1.\] The preimage $F^{-1}(\varphi)\cong\pp\rtimes\bbZ<\mcg$. This induces an injection \[g:\pi_1(M_\varphi)\hookrightarrow\mcg.\] $M_\varphi$ fibers over $S^1$ with fiber $\Sigma_g$. 
     As long as $H_2(M_\varphi;\bbZ)\geq 2$, the theory of the Thurston norm \cite{thurston} implies that $M_\varphi$ fibers over $S^1$ with fiber $\Sigma_h$ for infinitely many $h$. (These $h$ correspond to integer points in the cone over a fibered face of the unit ball in the Thurston norm.) Fiberings of the form
     \[
    \begin{tikzcd}
    \Sigma_h\arrow{r}{}& M_\varphi\arrow{d}\\
    & S^1
    \end{tikzcd}
    \]
    give injections $i_h:\pi_1(\Sigma_h)\to \pi_1(M_\varphi)$. The image of the composition \[g\circ i_h:\Sigma_h\hookrightarrow\mcg\] is a surface subgroup of $\mcg$. This subgroup is not necessarily normal in $\mcg$. Using the fibered faces of the unit ball in the Thurston norm, we can find multiple (non-normal) copies of $\pp$ in $\mcg$.\\

    \end{exm}

\section{An new proof that Out$(\extmcg)$ is trivial.}\label{burnsec}

\begin{corn}[\ref{maincorollary} (Ivanov-McCarthy's Theorem).]
Let $g\geq 3$. Then $\text{Out}(\extmcg)$ is trivial.
\end{corn}

Theorem \ref{maintheorem} together with the following classical theorem of Burnside implies Corollary \ref{maincorollary}. A group $G$ is \emph{complete} if it is centerless and every automorphism is inner, i.e. $\text{Aut}(G)\cong \text{Inn}(G)\cong G$. A subgroup $H< G$ is \emph{characteristic} if $H$ is invariant under all automorphisms of $G$.
\begin{thm}[(Burnside \cite{burnside})]\label{burn} A centerless group $G$ is characteristic in its automorphism group if and only if $\text{Aut}(G)$ is complete.
\end{thm}

\begin{proof}[Proof of $(\Rightarrow)$ for Theorem \ref{burn}] Suppose that $G$ is centerless and characteristic in $\text{Aut}(G)$. Let $\phi\in \text{Aut}(\text{Aut}(G))$ and let $g\in G$. There is a homomorphism \[i:G\rightarrow \text{Inn}(G)\] given by \[i(g)(h)=ghg^{-1}\] for any $h\in G$. The homomorphism $i$ is an isomorphism because $G$ is centerless. Additionally, because $G$ is characteristic, $\phi$ restricts to an automorphism of $\text{Inn}(G)\cong G$. Define \[\bar{\phi}:G\to G\] by \[i(\bar{\phi}(g)):=\phi(i(g)).\] To show that $\text{Aut}(\text{Aut}(G))=\text{Aut}(G)$, it suffices to show that \[\phi(\psi)=\bar{\phi}\circ\psi\circ \bar{\phi}^{-1}
\] for any $\psi\in\text{Aut}(G)$. Notice that: \[
\phi(i(\psi(g)))=i(\bar{\phi}(\psi(g))).\]

On the other hand \begin{eqnarray*}
\phi(i(\psi(g)))&=& \phi(\psi\circ i(g)\circ \psi^{-1})\\ &=& \phi(\psi)\circ i(\bar{\phi}(g))\circ \phi(\psi)^{-1}\\&=& i(\phi(\psi)(\bar{\phi}(g)).
\end{eqnarray*}  Because $i$ is an isomorphism we can equate \[\phi(\psi)(\bar{\phi}(g))=\bar{\phi}(\psi(g))\] for any $g\in G$. Therefore, \begin{eqnarray*}\phi(\psi)\circ\bar{\phi}&=&\bar{\phi}\circ\psi.\\  \end{eqnarray*} As such, \begin{eqnarray*}
 \phi(\psi)&=&\bar{\phi}\circ\psi\circ\bar{\phi}^{-1}.\end{eqnarray*}
\end{proof}

\begin{proof}[Proof of corollary \ref{maincorollary}] By Theorem \ref{maintheorem}, $P$ is characteristic in $\mcg$. By the Dehn-Nielsen-Baer theorem (see e.g. \cite{primer} Th. 8.1) it follows that $\text{Aut}(P)\cong\extmcg$. To prove the corollary, it suffices to show that $\mcg$ is characteristic in $\extmcg$. Notice that: \begin{eqnarray*}\bbZ/2\bbZ&\cong & \frac{\extmcg}{[\extmcg,\extmcg]}\cong H_1(\extmcg;\bbZ)\hspace{.5cm}\text{ and } \\ \bbZ/2\bbZ& \cong & \frac{\extmcg}{\mcg}.\end{eqnarray*} For further details on these quotients see \cite[Th. 5.2 and Ch. 8]{primer}. Because the quotient $\extmcg/\mcg$ is abelian, $[\extmcg,\extmcg]\subset\mcg$. Further, because the quotients are isomorphic and finite, it follows that $\mcg$ is equal to the commutator subgroup of $\extmcg$. Therefore $\mcg$ is characteristic, and $\text{Out}(\extmcg)\cong 1$.
\end{proof}

\bibliographystyle{alpha}
\bibliography{AlgCharPointPush}

\end{document}